\documentclass[12pt,reqno]{amsart}
\usepackage{geometry}                
\usepackage{amsmath,amssymb}
\usepackage{graphicx}
\usepackage{amssymb,latexsym, amsmath}
\usepackage{amsfonts,amsthm}
\usepackage{amscd}
\usepackage{mathrsfs}
\usepackage{hyperref}
\usepackage{eucal}
\usepackage{epstopdf}
\usepackage{amsgen}
\usepackage{xspace}
\usepackage{verbatim}
\usepackage{stmaryrd}
\usepackage{enumitem}
\newlist{steps}{enumerate}{1}
\setlist[steps, 1]{label = Step \arabic*:}

\newcommand{\gd}{\Delta}

\newcommand{\inpt}[1]{\langle #1 \rangle}

\newcommand{\gw}{\Omega}

\newcommand{\ga}{\gamma}

\newcommand{\G}{\Gamma}
\newcommand{\gl}{\lambda}

\newcommand{\gk}{\kappa}
\newcommand{\om}{\omega}

\newcommand{\nb}{\nabla}
\newcommand{\vp}{\varphi}
\newcommand{\ve}{\varepsilon}
\newcommand{\pdr}{\partial}

\newcommand{\tup}{\textup}

\newcommand{\beq}{\begin{equation}}
\newcommand{\eeq}{\end{equation}}
\newcommand{\bea}{\begin{align}}
\newcommand{\eea}{\end{align}}
\newcommand{\bthm}{\begin{theorem}}
\newcommand{\ethm}{\end{theorem}}
\newcommand{\bpr}{\begin{proof}}
\newcommand{\epr}{\end{proof}}
\newcommand{\bcl}{\begin{corollary}}
\newcommand{\ecl}{\end{corollary}}
\newcommand{\bpn}{\begin{proposition}}
\newcommand{\epn}{\end{proposition}}
\newcommand{\bre}{\begin{remark}}
\newcommand{\ere}{\end{remark}}
\newcommand{\bdf}{\begin{definition}}
\newcommand{\edf}{\end{definition}}
\newcommand{\bss}{\begin{align*}}
\newcommand{\ess}{\end{align*}}

\newcommand{\bl}{\label}

\newcommand{\zsw}{\theta_s \omega}
\newcommand{\ctw}{\theta_{-t} \omega}

\newtheorem{theorem}{Theorem}[section]
\newtheorem{corollary}[theorem]{Corollary}

\newtheorem{lemma}[theorem]{Lemma}
\newtheorem{proposition}[theorem]{Proposition}

\theoremstyle{definition}
\newtheorem{definition}[theorem]{Definition}
\theoremstyle{remark}
\newtheorem{remark}{Remark}

\numberwithin{equation}{section}

\begin{document}

\title[Stochastic Hindmarsh-Rose Equations]{Random Attractor for Stochastic Hindmarsh-Rose Equations with Additive Noise}

\author[C. Phan]{Chi Phan}
\address{Department of Mathematics and Statistics, University of South Florida, Tampa, FL 33620, USA}
\email{chi@mail.usf.edu}
\thanks{}

\author[Y. You]{Yuncheng You}
\address{Department of Mathematics and Statistics, University of South Florida, Tampa, FL 33620, USA}
\email{you@mail.usf.edu}
\thanks{}





\begin{abstract}
For stochastic Hindmarsh-Rose equations with additive noises in the study of neurodynamics, the longtime and global pullback dynamics on a two-dimensional bounded domain is explored in this work. Using the additive transformation and by the sharp uniform estimates, we proved the pullback absorbing and the pullback asymptotically compact characteristics of the Hindmarsh-Rose random dynamical system in the $L^2$ Hilbert space. It shows the existence of a random attractor for this random dynamical system.
\end{abstract}

\maketitle

\section{\textbf{Introduction}}
The Hindmarsh-Rose equations for neuronal spiking-bursting of the intracellular membrane potential observed in experiments was originally proposed in \cite{HR1, HR2}. This mathematical model composed of three coupled nonlinear ordinary differential equations has been studied through numerical simulations and mathematical analysis in recent years, cf. \cite{HR1, HR2, IG, MFL, SPH, Su} and the references therein. It exhibits rich and interesting spatial-temporal bursting patterns, especially chaotic bursting and dynamics as well as complex bifurcations. 

Very recently, we have proved the existence of a random attractor for the stochastic Hindmarsh-Rose equations with multiplicative noise in \cite{Phan}.

In this work, we shall study the longtime random dynamics in terms of the existence of a random attractor for the diffusive Hindmarsh-Rose equations driven by the additive noise,
\begin{align}
du & = d_1 \gd u \,dt+  (\vp (u) + v - z + J)\, dt + h_1(x)\, dW_1, \bl{suq1} \\
dv & = d_2 \gd v \,dt+ (\psi (u) - v) \,dt +  h_2(x)\,  dW_2, \bl{svq1} \\
dz & = d_3 \gd z\, dt + (q (u - c) - rz)\, dt + h_3(x) \,dW_3, \bl{szq1}
\end{align}
for $t > \tau,\; x \in \gw \subset \mathbb{R}^{n}$ ($n \leq 2$), where the nonlinear terms are

$$
\vp (u) = au^2 - bu^3 \quad \text{and} \quad \psi (u) = \alpha - \beta u^2.
$$
We impose the homogeneous Neumann boundary condition
\begin{equation} \label{nbc3}
    \frac{\partial u}{\partial \nu} (t, x) = 0, \; \; \frac{\partial v}{\partial \nu} (t, x)= 0, \; \; \frac{\partial z}{\partial \nu} (t, x)= 0,\quad  t > \tau \in \mathbb{R} ,  \; x \in \partial \gw,
\end{equation}
and an initial condition
\begin{equation} \bl{inc3}
u(\tau, x) = u_0 (x), \; v(\tau, x) = v_0 (x), \; z(\tau, x) = z_0 (x), \quad \tau \in \mathbb{R}, \; x \in \gw.
\end{equation}
The parameters $d_1, d_2, d_3, a, b, \alpha, \beta, q, r$, and $J$ are arbitrary positive constants and $c \in \mathbb{R}$ which is the reference value for the membrane potential of a neuron cell. Moreover, $\{h_i(x): i=1, 2, 3\} \subset W^{2,4}(\gw)$ are given functions and $W(t) = \{W_1(t), W_2(t), W_3(t)\}$, where $W_i (t), i = 1, 2, 3$, are independent, two-sided, real-valued standard Wiener processes on an underlying probability space $(\mathfrak{Q}, \mathcal{F}, P)$ to be specified later. 

In this system \eqref{suq1}-\eqref{szq1}, the variable $u(t,x)$ is the membrane electric potential of a neuronal cell, the variable $v(t, x)$ represents the transport rate of the ions of sodium and potassium through the fast ion channels and is called the spiking variable, while $z(t, x)$ is the bursting variable, which corresponds to the transport rate across the neuronal cell membrane through slow channels of calcium and other ions correlated to the bursting phenomenon.

In 1982-1984, J.L. Hindmarsh and R.M. Rose developed the mathematical model of ordinary differential equations to describe neuronal dynamics:
\begin{equation} \label{HR}
\begin{split}
\frac{du}{dt} & = au^2 - bu^3 + v - z + J,  \\
\frac{dv}{dt} & = \alpha - \beta u^2  - v,  \\
\frac{dz}{dt} & =  q (u - u_R) - rz.
\end{split}
\end{equation}
This model characterizes the phenomena of synaptic bursting and chaotic bursting. 

Neuronal signals are short electrical pulses known as spike or action potential. Neurons may display bursts of alternating phases of rapid firing spikes and then quiescence. Bursting patterns occur in various bio-systems such as pituitary melanotropic gland, thalamic neurons, respiratory pacemaker neurons, and insulin-secreting pancreatic $\beta$-cells, cf. \cite{BRS, CK,CS, HR2}. Mathematical neuron models on bursting behavior have been investigated mainly by using bifurcation theory and numerical simulations, cf. \cite{BB, ET, EI, MFL, Ri, SPH, Tr, WS, Su}.

The four-dimensional Hodgkin-Huxley equations \cite{HH}, which is highly nonlinear if without simplification, and the two-dimensional FitzHugh-Nagumo equations \cite{FH} are well-known models for excitable neurons with many studies but not quite suitable to characterize the neuronal chaotic bursting and chaotic dynamics. The 2D nature of FitzHugh-Nagumo equations prevents that model to generate any chaotic solutions.

Neurons communicate and coordinate actions through regular synaptic coupling or diffusive synchronizing coupling in neuroscience. Synaptic coupling has to reach certain threshold for release of quantal vesicles \cite{DJ, Ru, SC}, while the chaotic coupling exhibited in the current simulations and analysis of this Hindmarsh-Rose model \eqref{HR} shows more rapid and effective synchronization of neurons due to \emph{lower threshold} than the synaptic coupling \cite{Tr, Su}. Moreover, the Hindmarsh-Rose model allows varying interspike-interval when the parameters vary. Therefore, this 3D Hindmarsh-Rose model \eqref{HR} is a suitable choice for the investigation of both regular bursting and chaotic bursting.

Recently it has been proved by the two authors of this paper and J. Su in \cite{PYS} that there exist global attractors for the diffusive and partly diffusive Hindmarsh-Rose equations. We have also shown in \cite{Phan} that there exists a random attractor for the stochastic Hindmarsh-Rose equations with multiplicative noise. 

With the presence of additive independent white noises in a random environment as well as the diffusion of ions and membrane potential included in the Hindmarsh-Rose neuron model, here in this paper we shall study the longtime and global dynamics of pullback solutions of the random dynamical system generated by \eqref{suq1}-\eqref{szq1}, focusing on the existence of a random attractor through the approach of the additive transformation by means of the Ornstein-Uhlenbeck processes. 

The rest of Section 1 is the formulation of the stochastic system \eqref{suq1}-\eqref{szq1} with some basic concepts and results in the theory of random dynamical systems. In Section 2, the global existence of pullback weak solutions is established together with the pullback absorbing property of the stochastic Hindmarsh-Rose cocycle in the $L^{2}$ space. In Section 3, we shall prove the pullback asymptotical compactness and the main result on the existence of a random attractor for the diffusive Hindmarsh-Rose random dynamical system with the additive noise. 

\subsection{\textbf{Preliminaries}}
To study the stochastic and global dynamics of differential equations in the asymptotically long run, we recall preliminary concepts for random dynamical systems, cf. \cite{Ar, Ch, CDF, CF, EZ, Ok, SH, Tm, W12, W14, W19, Y12, Y14, Y17}. Let $(\mathfrak{Q}, \mathcal{F}, P)$ be a probability space and let $X$ be a real Banach space. 

\begin{definition} \label{sc1}
	$(\mathfrak{Q}, \mathcal{F}, P, \{{\theta_t}\}_{t \in \mathbb{R}})$ is said to be a $\textit{metric dynamical system}$, which is briefly called MDS, if $(\mathfrak{Q}, \mathcal{F}, P)$ is a probability space with a time-parametrized mapping $\theta_t$ and the following conditions are satisfied:
	
	(i) The mapping $\theta_t : \mathfrak{Q} \to \mathfrak{Q}$ is $\mathcal{F}$-measurable, $t \in \mathbb{R}$.
	
	(ii) $\theta_0$ is the identity on $\mathfrak{Q}$.
	
	(iii) $\theta_{t + s} = \theta_t \circ \theta_s$ for all $t, s \in \mathbb{R}$.
	
	(iv) $\theta_t$ is probability invariant, meaning $\theta_t P = P$ for all $t \in \mathbb{R}$.
	
	\noindent Here $(\theta_t P)(S) = P(\theta_t S)$ for any $S \in \mathcal{F}$.
\end{definition}
	Denote by $\mathscr{B} (X)$ the $\sigma$-algebra of all Borel sets in a Banach space $X$
	
\begin{definition} \label{sc2}
	A continuous $\textit{random dynamical system}$ briefly called a cocycle on $X$ over $(\mathfrak{Q}, \mathcal{F}, P, \{{\theta_t}\}_{t \in \mathbb{R}})$ is a mapping
	$$
	\varphi (t, \omega, x) : [0, \infty) \times  \mathfrak{Q} \times X \to X,
	$$
	which is $(\mathscr{B} (\mathbb{R^+}) \otimes \mathcal{F} \otimes \mathscr{B} (X), \mathscr{B} (X))$-measurable and satisfies the following conditions for every $\omega$ in $\mathfrak{Q}$:
	
	(i) $\varphi (0, \omega, \cdot)$ is the identity operator on $X$.
	
	(ii) The cocycle property holds:
	$$
	\varphi (t + s, \omega, \cdot) = \varphi (t, \theta_s \omega, \varphi (s, \omega, \cdot)), \quad  \text{for all} \; \; t, s \geq 0.
	$$
	
	(iii) The mapping $\varphi (\cdot, \omega, \cdot) : [0, \infty) \times X \to X$ is continuous.
\end{definition}

\begin{definition} \label{sc3}
	A set-valued function $B : \mathfrak{Q} \to 2^X$ is called a $\textit{random set}$ in $X$ if its graph $\{(\omega, x) : x \in B(\omega)\} \subset \mathfrak{Q} \times X$ is an element of the product $\sigma$-algebra $\mathcal{F} \otimes \mathscr{B} (X)$. A $\textit{bounded}$ random set $B(\omega) \subset X$ means that there is a random variable $r(\omega) \in [0, \infty), \omega \in \mathfrak{Q}$, such that $\interleave {B(\omega)} \interleave := \sup_{x \in B(\omega)} \|x\| \leq r(\omega)$ for all $\omega \in \mathfrak{Q}$. A bounded random set is called $\textit{tempered}$ with respect to $\{\theta_t\}_{t \in \mathbb{R}}$ on $(\mathfrak{Q}, \mathcal{F}, P)$ if for $\omega \in \mathfrak{Q}$ and for any constant $\ve > 0$,
	$$
	\lim_{t \to \infty} e^{-\ve t} \interleave B(\theta_{-t} \omega)\interleave = 0.
	$$
	A random set $S(\omega) \subset X$ is called \emph{compact} (respectively \emph{precompact}) if for each $\omega \in \mathfrak{Q}$ the set $S(\omega)$ is a compact (resoectively precompact) set in $X$.
\end{definition}

\begin{definition} \label{sc4}
	A random variable $R : (\mathfrak{Q}, \mathcal{F}, P) \to (0, \infty)$ is called $ \textit{tempered with}$ $\textit{respect to a metric dynamical system}$ $\{\theta_t\}_{t \in \mathbb{R}}$ on $(\mathfrak{Q}, \mathcal{F}, P)$, if for all $\om \in \gw$,
	$$
	\lim_{t \to -\infty} \frac{1}{t} \; \text{log} \; R(\theta_t \omega) = 0.
	$$
\end{definition}

We use $\mathscr{D}_X$ to denote an inclusion-closed family of random sets in $X$, which is called a $\textit{universe}$ \cite{Ar, Ch}. In this work, $\mathscr{D}_X$ is the universe of all the tempered random sets in a specified space $X$.

\begin{definition} \label{sc5}
	 A random set $K \in \mathscr{D}_X$ is called a $\textit{pullback absorbing set}$ with respect to a random dynamical system $\varphi$ over the MDS $(\mathfrak{Q}, \mathcal{F}, P, \{{\theta_t}\}_{t \in \mathbb{R}})$, if for any bounded random set $B \in \mathscr{D}_X$ and  $\omega \in \mathfrak{Q}$ there exists a finite time $T_B (\omega) > 0$ such that
	$$
	\varphi (t, \theta_{-t} \omega, B (\theta_{-t} \omega)) \subset K (\omega), \quad \text{for all} \; \; t \geq T_B (\omega).
	$$
\end{definition}

\begin{definition} \label{sc6}
	A continuous random dynamical system $\varphi$ is \emph{pullback asymptotically compact} with respect to $\mathscr{D}_X$ , if for $\omega \in \mathfrak{Q}$,
	$$
	\{\varphi (t_m, \theta_{-t} \omega, x_m)\}^\infty_{m = 1} \; \text{has a convergent subsequence in} \; X,
	$$
	whenever $t_m \to \infty$ and $x_m \in B (\theta_{-t} \omega)$ for any given bounded random set $B \in \mathscr{D}_X$.
\end{definition}

\begin{definition} \label{sc7}
	A random set $\mathcal{A} \in \mathscr{D}_X$ is called a $\textit{random attractor}$ in $\mathscr{D}_X$ for a given random dynamical system $\varphi$ over the metric dynamical system $(\mathfrak{Q}, \mathcal{F}, P, \{{\theta_t}\}_{t \in \mathbb{R}})$, if the following conditions are satisfied for all $\om \in \gw$:
	
	(i) $\mathcal{A}$ is a compact random set.
	
	(ii) $\mathcal{A}$ is invariant in the sense that 
	$$
	\varphi (t, \omega, \mathcal{A}(\omega)) = \mathcal{A}(\theta_t \omega), \quad \text{for all} \;\; t \geq 0.
	$$
	
	(iii) $\mathcal{A}$  attracts every $B \in \mathscr{D}_X$ in the pullback sense that 
	$$
	\lim_{t \to \infty} dist_X (\varphi (t, \theta_{-t} \omega, B (\theta_{-t} \omega)), \mathcal{A} (\omega)) = 0,
	$$
	where $dist_X (\cdot, \cdot)$ is the Hausdorff semi-distance with respect to the $X$-norm. $\mathscr{D}_X$ is called the $\textit{basin}$ of attraction for the attractor $\mathcal{A}$.
\end{definition}	

The existence of random attractors for continuous random dynamical systems has been investigated by many authors, cf. \cite{Ar, BLW, Ch, CDF, CF, SH, Sm, Y14, Y17}.  The following theorem on the existence of random attractors will be used.

\begin{theorem} \bl{sc8}
	Given a Banach space $X$ and a family $\mathscr{D}_X$ of random sets in $X$, let $\varphi$ be a continuous random dynamical system on $X$ over the metric dynamical system $(\mathfrak{Q}, \mathcal{F}, P, \{{\theta_t}\}_{t \in \mathbb{R}})$. If the following two conditions are satisfied\textup{:}
	
	\textup{(i)} there exists a closed pullback absorbing set $K = \{K (\omega)\}_{\omega \in \mathfrak{Q}} \in \mathscr{D}_X$ for the cocycle $\varphi$,
	
	\textup{(ii)} the cocycle $\varphi$ is pullback asymptotically compact with respect to $\mathscr{D}_X$, 
	
	\noindent then there exists a unique random attractor $\mathcal{A} = \{\mathcal{A} (\omega)\}_{ \omega \in \mathfrak{Q}} \in \mathscr{D}_X$ for the random dynamical system $\varphi$.  The random attractor $\mathcal{A}$ is given by
	\beq \bl{omK}
	\mathcal{A} (\omega) = \bigcap_{\tau \geq 0} \; {\overline{\bigcup_{ t \geq \tau} \varphi (t, \theta_{-t} \omega, K (\theta_{-t} \omega))}}, \quad \om \in \mathfrak{Q}.
	\eeq
\end{theorem}

\begin{proof}
	The proof is seen in \cite{BLL, Ch, CF, SH} except the $\mathcal{F}$-measurability of the $\om$-limit set of $K(\theta_{-t} \om)$ in \eqref{omK} is shown in \cite[Theorem 2.9]{Y17}.
\end{proof}

\subsection{\textbf{Formulation and Random Environment}}
We now formulate the initial-boundary value problem \eqref{suq1}--\eqref{inc3} of the stochastic Hindmarsh-Rose equations with the additive noise in a framework of the Hilbert spaces
\beq \bl{srd1}
H = L^2 (\gw, \mathbb{R}^3) \quad \text{and} \quad E = H^1 (\gw, \mathbb{R}^3).
\eeq
The norm and inner-product of $H$ or $L^2 (\gw)$ will be denoted by $\| \cdot\|$ and $\inpt{\cdot, \cdot}$, respectively. The norm of space $E$ will be denoted by $\| \cdot\|_E$. The norm of $L^p (\gw) $ or $L^p (\gw, \mathbb{R}^3)$ will be denoted by $\| \cdot\|_{L^p}$ for $p \neq 2$. W use $| \cdot|$ to denote a vector norm in Euclidean spaces. A stochastic process will be denoted by $Y(t), Y_t$ or by $Y (t, \omega)$ to indicate the sample path, whichever is convenient depending on the context.

The nonpositive self-adjoint linear differential operator
\begin{equation} \label{srd2a}
A =
\begin{pmatrix}
d_1 \gd  & 0   & 0 \\[3pt]
0 & d_2 \gd  & 0 \\[3pt]
0 & 0 & d_3 \gd
\end{pmatrix}
: D(A) \rightarrow H,
\end{equation}
where 
\begin{equation*}
D(A) = \left\{(u, v, z) \in H^2 (\gw, \mathbb{R}^3): \frac{\partial u}{\partial \nu} = \frac{\partial v}{\partial \nu} = \frac{\partial z}{\partial \nu} = 0 \;\, \text{on} \; \, \partial \gw \right\}
\end{equation*}
is the generator of an analytic contraction $C_0$-semigroup $\{e^{At}\}_{t \geq 0}$ on the Hilbert space $H$. By the fact that $H^{1}(\gw) \hookrightarrow L^6(\gw)$ is a continuous Sobolev imbedding for space dimension $n \leq 3$, the nonlinear mapping 
\begin{equation} \label{srd3a}
f(u,v, z) =
\begin{pmatrix}
\vp (u) + v - z + J \\[4pt]
\psi (u) - v,  \\[4pt]
q (u - c) - rz
\end{pmatrix}
: E \longrightarrow H
\end{equation}
is locally Lipschitz continuous. Let $W(t) = \text{col}\, (W_1(t), W_2(t), W_3(t))$ and 
$$
	\Lambda (h) \ =
		\begin{pmatrix}
		a_1 h_1(x)  & 0   & 0 \\[3pt]
		0 & a_2 h_2(x)  & 0 \\[3pt]
		0 & 0 & a_3 h_3(x)
\end{pmatrix}.
$$
Then the initial-boundary value problem \eqref{suq1}--\eqref{inc3} is formulated into an initial value  problem of the following stochastic Hindmarsh-Rose evolutionary equation driven by the additive noise:
\begin{equation} \label{srd4a}
\begin{split}
dg &= A \,g \,dt + f(g)\, dt + \Lambda (h)\, dW, \quad t > \tau \in \mathbb{R},\\[3pt]
&g (\tau, \om, g_0) = g_0 = (u_0, v_0, z_0) \in H.
\end{split}
\end{equation}
The solutions of \eqref{srd4a} is denoted by 
$$
	g(t, \omega, g_0) = \text{col} \, (u(t, \cdot, \omega, g_0), v(t, \cdot, \omega, g_0), z(t, \cdot, \omega, g_0))
$$
where $(u, v, z)$ is the vector of solutions to the problem \eqref{suq1}-\eqref{inc3}, and dot stands for the hidden spatial variable $x$, and $\omega \in \mathfrak{Q}$.

Specifically assume that $\{W_i(t): i = 1, 2, 3\}_{t \in \mathbb{R}}$ are independent two-sided standard Wiener process (Brownian motion) in the canonical probability space $(\mathfrak{Q}, \mathcal{F}, P)$, where the sample space
\beq \bl{scrd5a}
\mathfrak{Q} = \{\om (t) = (\om_1 (t), \om_2 (t), \om_3 (t)) \in C(\mathbb{R}, \mathbb{R}^3) : \om (0) = 0\},
\eeq
the $\sigma$-algebra $\mathcal{F}$ is generated by the compact-open topology endowed in $\mathfrak{Q}$, and $P$ is the corresponding Wiener measure \cite{Ar, Ch, CF, Ok} on $\mathcal{F}$. Define a family of $P$-preserving time-shift transformations $\{\theta_t\}_{t \in \mathbb{R}}$ by
\beq \bl{srd6a}
(\theta_t\, \omega) (\cdot) = \omega (\cdot + t) - \omega (t), \quad \text{for} \; \; t \in \mathbb{R},\; \omega \in \mathfrak{Q}.
\eeq
Then $(\mathfrak{Q}, \mathcal{F}, P, \{\theta_t\}_{t \in \mathbb{R}})$ is a metric dynamical system and the stochastic process $\{W(t, \om) = \om(t): t \in \mathbb{R}, \om \in \mathfrak{Q}\}$ is a three-dimensional canonical Wiener process.

\begin{proposition}\cite{Ok} \label{sc9}
	The Wiener process $W(t)$ defined above has the asymptotically sublinear growth property,
	\beq \bl{srd7}
	\lim_{t \to \pm \infty} \frac{|W(t)|}{|t|} = 0, \quad \text{a.s.}
	\eeq
\end{proposition}	

For a given $\gk > 0$ to be specified, introduce the Ornstein-Uhlenbeck process $\G(\theta_t \om) = \text{col}\, (\G_1(\theta_t \om_1), \G_2(\theta_t \om_2), \G_3(\theta_t \om_3))$, which is defined by
\beq \bl{snd1}
\begin{split}
	&\G_i(t, \om_i) = - \kappa \int_{-\infty}^{t} e^{-\gk (t - s)} dW_i(s, \om_i) = - \gk \int_{-\infty}^{0} e^{\gk \xi} dW_i (t+\xi, \om_i) \\
        = - \gk \int_{-\infty}^{0} &e^{\gk s} dW_i (s, \theta_t \om_i) = - \gk \int_{-\infty}^0 e^{\gk s} (\theta_t \om_i)(s) ds = \G_i(0, \theta_t \om_i) := \G_i(\theta_t \om_i).
\end{split}
\eeq
The Ornstein-Uhlenbeck processes $\G_i (t, \om_i) = \G_i (\theta_t \om_i), i = 1, 2, 3,$ satisfy the scalar stochastic differential equation
\beq \bl{snd2}
	d\G_i = -\gk \G_i\,dt + dW_i, \quad \G(-\infty) = 0.
\eeq

Define $\G^h(\theta_t \om) = \text{col} \,(\G_1^h(\theta_t \om_1), \G_2^h(\theta_t \om_2), \G_3^h(\theta_t \om_3))$ to be the corresponding abstract Ornstein-Uhlenbeck process 
\beq \bl{snd3}
	\G_i^h(\theta_t \om_i) = h_i(x) \G_i(\theta_t \om_i), \quad 1 \leq i \leq 3.
\eeq
For any $p \geq 2$ and any $\gk > 0$, the Ornstein-Uhlenbeck process $\G(\theta_t \om)$ is tempered in $L^p( \mathbb{R},  \mathbb{R}^3)$. It means that for any $\ve >0$,
\beq \bl{snd4}
	\lim_{|t| \to \infty} e^{-\ve |t|} \,|\G(\theta_t \om)|^p = 0.
\eeq
Thus the abstract Ornstein-Uhlenbeck process $\G^h(\theta_t \om)$ satisfies the similar property: if $h_i \in L^p(\gw), 1 \leq i \leq 3$, then for any $\ve > 0$,
\beq \bl{snd6}
		\lim_{|t| \to \infty} e^{-\ve |t|} \,\|\G^h(\theta_t \om)\|_{L^p(\gw, \mathbb{R}^3)}^p = 0.
\eeq

\section{\textbf{Hindmarsh-Rose Cocycle and Pullback Absorbing Property}}

The first step to treat the stochastic PDE problem \eqref{suq1}--\eqref{inc3} is to convert the system to random PDE, which has random coefficients and random initial data, by the additive transformation:
\beq \bl{snd8}
\begin{split}
	&\, U(t, \om; \tau, g_0) = u(t, \cdot, \om, \tau, g_0) - \G_1^h(\theta_t \om_1),\\
	&\, V(t, \om; \tau, g_0) = v(t, \cdot, \om, \tau,  g_0) - \G_2^h(\theta_t \om_2),\\
	&\, Z(t, \om; \tau, g_0) = z(t, \cdot, \om, \tau, g_0) - \G_3^h(\theta_t \om_3),
\end{split}
\eeq
where $\om = (\om_1, \om_2, \om_3)$, and dot stands for the hidden spatial variable $x$. 

Then the initial-boundary value problem \eqref{suq1}--\eqref{inc3} is converted to the following system of random partial differential equations:
\begin{align}
\frac{\pdr U}{\pdr t} & = d_1 \gd U + d_1 \gd h_1 \G_1(\theta_t \om_1) +  a (U + \G_1^h(\theta_t \om_1))^2 - b (U + \G_1^h(\theta_t \om_1))^3 \nonumber \\
&\quad + (V + \G_2^h(\theta_t \om_2)) - (Z + \G_3^h(\theta_t \om_3)) + J + \gk \G_1^h(\theta_t \om_1), \bl{sUq1} \\
\frac{\pdr V}{\pdr t} & = d_2 \gd V + d_2  \gd h_2 \G_2(\theta_t \om_2) + \alpha - \beta (U + \G_1^h(\theta_t \om_1))^2 \nonumber \\
& \quad - (V + \G_2^h(\theta_t \om_2))+ \gk \G_2^h(\theta_t \om_2) , \bl{sVq1} \\
\frac{\pdr Z}{\pdr t} & = d_3 \gd Z + d_3 \gd h_3 \G_3(\theta_t \om_3) + q (U + \G_1^h(\theta_t \om_1) - c) \nonumber \\ 
&\quad - r (Z + \G_3^h(\theta_t \om_3)) + \gk \G_3^h(\theta_t \om_3)  , \bl{sZq1}
\end{align}
for $\om \in \mathfrak{Q}, \, t > \tau,\, x \in \gw \subset \mathbb{R}^{n}$ ($n \leq 2$), with the Neumann boundary condition
\begin{equation} \label{nbc4}
	\frac{\partial U}{\partial \nu} (t, x, \om) = 0, \; \frac{\partial V}{\partial \nu}(t, x, \om) = 0, \; \frac{\partial Z}{\partial \nu}(t, x, \om) = 0, \quad t \geq \tau \in \mathbb{R}, \;  x \in \partial \gw,
\end{equation}
and an initial condition 
\begin{equation} \bl{inc4}
(U, V, Z)(\tau, \om) = g_0 - \G^h(\theta_\tau \om) = (u_0 - \G_1^h(\theta_\tau \om_1),\; v_0 - \G_2^h(\theta_\tau \om_2),\; z_0 - \G_3^h(\theta_\tau \om_3)).
\end{equation}

The initial-boundary value problem \eqref{sUq1}-\eqref{inc4} can be written as an initial value problem of the pathwise non-autonomous random evolutionary equation
\beq \bl{sGq1}
	\begin{split}
	\frac{\pdr G}{\pdr t}  = &\,  A G  +  \mathscr{F} (G, \theta_t\, \om), \quad t \geq \tau \in \mathbb{R}, \; \om \in \mathfrak{Q}, \\[3pt]
	&G (0, \om; \tau, g_0) = g_0 = (u_0, v_0, z_0) \in H.
	\end{split}
\eeq 
We define the weak solution of the initial value problem \eqref{sGq1},
\beq \bl{snd9}
	G (t, \om; \tau, g_0) = (U(t, \om; \tau, g_0), V(t, \om; \tau, g_0), Z(t, \om; \tau, g_0)),
\eeq
to be the weak solution of the nonautonomous initial-boundary problem \eqref{sUq1}-\eqref{inc4}, specified in \cite[Definition 2.1]{Y12}. 

By conducting estimates on the Galerkin approximate solutions and through the compactness argument outlined in \cite[Chapter II and XV]{CV} with some adaptations, we can prove the local existence and uniqueness of the weal solution $G(t, \om) = G(t, \om; \tau, g_0)$ in the space $H$ on a time interval $[\tau, T_{\text{max}}(\tau, \om, g_0))$ for some $\tau  < T_{\text{max}}(\tau, \om, g_0) \leq \infty$, and the solution continuously depends on the initial data. Further by the parabolic regularity \cite[Theorem 48.5]{SY}, every weak solution becomes a strong solution in the space $E$ for $t > \tau$ in the existence interval and has the regularity property 
\beq 
	G \in C([\tau, T_{max}), H) \cap C^1 ((\tau, T_{max}), H) \cap L^2_{loc} ([\tau, T_{max}), E).
\eeq

\subsection{\textbf{Global Existence of Pullback Solutions}}

The converted system of random partial differential equations \eqref{sUq1}-\eqref{sZq1} is non-autonomous by nature and we shall deal with the pullback weak solutions to investigate the random dynamics.

\begin{lemma}
	For any $\tau \in \mathbb{R},\, \om \in \mathfrak{Q}$, and any given initial data $g_0 = (u_0, v_0, z_0) \in H$, the weak solution $G(t,\om;\tau,g_0)$ defined in \eqref{snd9} of the initial boundary-problem of the random PDE \eqref{sUq1}-\eqref{inc4} uniquely exists on $[\tau,\infty)$. Consequently, the weak solution $(u,v,z)(t, \theta_\tau \om; \tau, g_0) = G(t, \theta_\tau \om; \tau, g_0) + \G^h (\theta_t \om)$ of the original problem \eqref{suq1}--\eqref{inc3} uniquely exists on $[\tau,\infty)$ and continuously depends on the initial data.
\end{lemma}

\begin{proof}
	Take the $H$ inner-products $\langle \eqref{sUq1}, c_1 U(t) \rangle$, $\langle \eqref{sVq1}, V(t) \rangle$ and $\langle \eqref{sZq1}, Z(t) \rangle$ with a constant $c_1 > 0$ to be specified later and then sum up the resulting equalities. Recall that $U = u - \G_1^h$, $V = v - \G_2^h$ and $Z = z - \G_3^h$. We obtain
	\beq \bl{snd11}
	\begin{split}
		&\frac{1}{2}\frac{d}{dt} \left(c_1 \|U\|^2+ \|V\|^2 + \|Z\|^2\right) + \left(c_1 d_1 \|\nb U\|^2 + d_2\|\nb V\|^2 + d_3\|\nb Z\|^2\right) \\[6pt]
		= &\, \int_{\gw} c_1 U \left[d_1 \gd h_1 \G_1(\theta_t \om_1) + k \G_1^h(\theta_t \om_1) + \G_2^h(\theta_t \om_2) - \G_3^h(\theta_t \om_3)\right]\, dx\\[2pt]
		+&\, \int_{\gw} V \left[d_2 \gd h_2 \G_2(\theta_t \om_2) - \G_2^h(\theta_t \om_2) + k \G_2^h(\theta_t \om_2)\right]\, dx\\[2pt]
		+&\, \int_{\gw} Z \left[d_3 \gd h_3 \G_3(\theta_t \om_3) + q \G_1^h(\theta_t \om_1) - r \G_3^h(\theta_t \om_3) + k \G_3^h(\theta_t \om_3)\right]\, dx\\[2pt]
		+&\,  \int_{\gw} \left[c_1 UV - c_1 ZU - V^2 + q(U - c)Z - rZ^2 + c_1 J U + \alpha V\right]\, dx\\[2pt]
		+&\, \int_{\gw} \left\{(c_1 au^3 - c_1 bu^4 - \beta vu^2) + [c_1 \G_1^h(\theta_t \om_1)(bu^3 - au^2) + \beta \G_2^h(\theta_t \om_2) u^2]\right\} dx.
	\end{split}
	\eeq
	For the first three integral terms on the right-hand side of equality \eqref{snd11}, we have
	
	\beq \bl{snd12}
	\begin{split}
		&\, \int_{\gw} c_1 U \left[d_1 \gd h_1 \G_1(\theta_t \om_1) + k \G_1^h(\theta_t \om_1) + \G_2^h(\theta_t \om_2) - \G_3^h(\theta_t \om_3)\right]\, dx\\[7pt]
		&\, + \int_{\gw} V \left[d_2 \gd h_2 \G_2(\theta_t \om_2) - \G_2^h(\theta_t \om_2) + k \G_2^h(\theta_t \om_2)\right]\, dx\\[7pt]
		&\, + \int_{\gw} Z \left[d_3 \gd h_3 \G_3(\theta_t \om_3) + q \G_1^h(\theta_t \om_1) - r \G_3^h(\theta_t \om_3) + k \G_3^h(\theta_t \om_3)\right]\, dx\\[7pt]
		\leq &\; c(h) |\G(\theta_t \om)|^2 + \frac{c_1^2}{2} \int_{\gw} U^2 dx + \frac{1}{12} \int_{\gw} V^2 dx + \frac{r}{6} \int_{\gw} Z^2 dx,
	\end{split}
	\eeq
	where $c(h) > 0$ is a constant depending on the functions $h(x) = (h_1 (x), h_2(x), h_3(x))$.  Note that 
	$$
	u^4 = \left[\left(U + \G_1^h(\theta_t \om_1)\right)^2\right]^2 \leq \left[2\left(u^2 + \left(\G_1^h(\theta_t \om_1)\right)^2\right)\right]^2 \leq 8\left[u^4 + \left(\G_1^h(\theta_t \om_1)\right)^4\right].
	$$
	The 5th integral term on the right-hand side of \eqref{snd11} is
	\beq \bl{snd13}
	\begin{split}
		&\, \int_{\gw} (c_1 au^3 - c_1 bu^4 - \beta vu^2)\,dx + c_1\int_{\gw} \G_1^h(\theta_t \om_1)(bu^3 - au^2) dx + \beta \int_{\gw} \G_2^h(\theta_t \om_2) u^2 dx \\[7pt]
		= & \int_{\gw} \left[(c_1 a + c_1 b \G_1^h(\theta_t \om_1)) u^3 - c_1 bu^4 - \beta vu^2 + \left(\beta \G_2^h(\theta_t \om_2) - c_1 a \G_1^h(\theta_t \om_1)\right) u^2\right]\, dx\\[7pt]
		\leq & \int_{\gw} \left[\frac{3}{4} u^4 + \frac{1}{4} \left(c_1 a + c_1 b \G_1^h(\theta_t \om_1)\right)^4 - c_1 bu^4 \right] dx  \\[7pt]
		+  &\, \int_{\gw} \left[ 2 \beta^2 u^4  +  \frac{v^2}{8} + \left(\beta \G_2^h(\theta_t \om_2) - c_1 a \G_1^h(\theta_t \om_1)\right)^2 + \frac{u^4}{4} \right] \, dx.
	\end{split}
	\eeq
	Choose the positive constant in \eqref{snd11} and \eqref{snd13} to be 
	$$
		c_1 = \frac{1}{b} \left(2 \beta^2 + \frac{11}{8}\right)
	$$
	so that 
	$$
		\int_{\gw} (-c_1 bu^4 + 2 \beta^2 u^4)\, dx \leq -\frac{11}{8} \int_{\gw} u^4 \,dx.
	$$
	Then \eqref{snd13} becomes 
	
	\beq \bl{snd14}
	\begin{split}
		&\, \int_{\gw} (c_1 au^3 - c_1 bu^4 - \beta vu^2)\,dx + c_1\int_{\gw} \G_1^h(\theta_t \om_1)(bu^3 - au^2)\,dx  + \beta \int_{\gw} \G_2^h(\theta_t \om_2) u^2\,dx \\[2pt]
		&\, \leq - \frac{3}{8} \int_{\gw}  u^4\, dx + \frac{1}{4} \int_{\gw} 8 \left[(c_1 a)^4 + (c_1 b)^4(\G_1^h(\theta_t \om_1))^4\right]\, dx + \int_{\gw} \frac{v^2}{8}\,dx \\[2pt]
		&\quad+ \int_{\gw} 2\left[\beta^2 (\G_2^h(\theta_t \om_2))^2 + (c_1 a_1)^2 (\G_1^h(\theta_t \om_1))^2\right]\,dx\\[5pt]
		&\, \leq - \frac{3}{8} \int_{\gw} \left[U + \G_1^h(\theta_t \om_1)\right]^4 \,dx + \frac{1}{8} \int_{\gw} \left[V + \G_2^h(\theta_t \om_2)\right]^2 \,dx + 2(c_1 a)^4 |\gw| \\[2pt]
		&\quad + 2(c_1 b)^4 \int_{\gw} \left(\G_1^h(\theta_t \om_1)\right)^4 \,dx + [2\beta^2 +(c_1 a)^2]\, \|\G^h(\theta_t \om)\|^2\\[2pt]
		&\, \leq -3 \int_{\gw} \left[U^4 + (\G_1^h(\theta_t \om_1))^4\right]\,dx + \frac{1}{4} \int_{\gw} \left[V^2 + (\G_2^h(\theta_t \om_2))^2\right] \,dx + 2(c_1 a)^4 |\gw| \\[7pt]
		&\quad+ 2(c_1 b)^4 \|\G^h(\theta_t \om)\|_{L^4}^4 +  [2\beta^2 +(c_1 a)^2]\, \|\G^h(\theta_t \om)\|^2\\[5pt]
		&\, \leq -3 \int_{\gw} U^4\,dx - 3 \|\G^h(\theta_t \om)\|_{L^4}^4  + \frac{1}{4} \int_{\gw} V^2 \,dx   \\[1pt]
		&\quad + \frac{1}{4}  \|\G^h(\theta_t \om)\|^2 + 2(c_1 b)^4 \|\G^h(\theta_t \om)\|_{L^4}^4   + [2\beta^2 +(c_1 a)^2]\, \|\G^h(\theta_t \om)\|^2 + 2(c_1 a)^4 |\gw|\\[5pt]
		&\, \leq -3 \int_{\gw} U^4\,dx + \frac{1}{4} \int_{\gw} V^2 \,dx + 2(c_1 b)^4 \|\G^h(\theta_t \om)\|_{L^4}^4\\[2pt]
		&\quad + \left[2\beta^2 +(c_1 a)^2 + \frac{1}{4} \right]\|\G^h(\theta_t \om)\|^2 + 2(c_1 a)^4 |\gw|.
	\end{split}
	\eeq
	Next, the 4th integral term in \eqref{snd11} is estimated,
	\begin{gather*}
		\int_{\gw} [c_1 UV - c_1 ZU - V^2  + c_1 J U + \alpha V + q(U - c)Z - rZ^2]\, dx \\[5pt]
		\leq \int_{\gw} \left[3 c_1^2 U^2 + \frac{V^2}{12} + \frac{3 c_1^2}{2r} U^2 + \frac{r}{6} Z^2 - V^2 + \frac{c_1^2}{2} U^2 + \frac{J^2}{2} + 3 \alpha^2 + \frac{V^2}{12} \right.  \\[5pt]
		\left. + \left(\frac{3q^2}{r} (U^2 + c^2) + \frac{r}{6} Z^2 \right) - rZ^2 \right] dx.
	\end{gather*}
	Collect all the integral terms with $U^2$ involved from the above inequality to obtain
	
	\begin{equation}  \bl{snd15}
	\begin{split}
		\int_{\gw} &\left(\frac{c_1^2}{2} + 3 c_1^2 + \frac{3 c_1^2}{2r} + \frac{c_1^2}{2} + \frac{3q^2}{r} \right) U^2\, dx = \int_{\gw} \left(4 c_1^2 + \frac{3 c_1^2}{2r} + \frac{3q^2}{r}\right) U^2\,dx\\[3pt]
		&\, \leq \int_{\gw} U^4 \,dx + \left(4 c_1^2 + \frac{3 c_1^2}{2r} + \frac{3q^2}{r}\right)^2 |\gw|.
	\end{split}
	\end{equation}
Assemble all the estimate \eqref{snd12}-\eqref{snd15} into \eqref{snd11}. Then we get
	\begin{equation*}
	\begin{split}
		&\frac{1}{2}\frac{d}{dt} \left(c_1 \|U\|^2+ \|V\|^2 + \|Z\|^2\right) + \left(c_1 d_1 \|\nb U\|^2 + d_2\|\nb V\|^2 + d_3\|\nb Z\|^2\right) \\[5pt]
		\leq &\,  \int_{\gw} (1-3) U^4 dx + \int_{\gw} \left(\frac{1}{12} - 1 + \frac{1}{12} + \frac{1}{12} + \frac{1}{4} \right) V^2 dx + \int_{\gw} \left(\frac{r}{6} + \frac{r}{6} + \frac{r}{6} - r \right) Z^2 dx \\[2pt]
		&\, + c(h) |\G(\theta_t \om)|^2 + 2(c_1 b)^4 \|\G^h(\theta_t \om)\|_{L^4}^4 + \left[2\beta^2 +(c_1 a)^2 + \frac{1}{4} \right]\, \|\G^h(\theta_t \om)\|^2\\
		&\, + \left[2(c_1 a)^4 + \frac{J^2}{2} + 3 \alpha^2 + \frac{3q^2c^2}{r} + \left(4 c_1^2 + \frac{3 c_1^2}{2r} + \frac{3q^2}{r}\right)^2\right]|\gw|\\
		\leq &\, \int_{\gw} -2 U^4 \,dx - \int_{\gw} \frac{V^2}{2}\,dx - \int_{\gw} \frac{r}{2} Z^2\,dx + c(\ga) |\G(\theta_t \om)|^2 + 2(c_1 b)^4 \|\G^h(\theta_t \om)\|_{L^4}^4 \\[3pt]
		&\, + \left[2\beta^2 +(c_1 a)^2 + \frac{1}{4} \right]\|\G^h(\theta_t \om)\|^2 + N |\gw|,
	\end{split}
	\end{equation*}
	where
	$$
	N = \left[2(c_1 a)^4 + \frac{J^2}{2} + 3 \alpha^2 + \frac{3q^2c^2}{r} + \left(4 c_1^2 + \frac{3 c_1^2}{2r} + \frac{3q^2}{r}\right)^2\right].
	$$
Let $d = \text{min}\,\{d_1, d_2, d_3\}$. It follows that
	\beq \bl{snd16}
	\begin{split}
		&\, \frac{d}{dt} \left(c_1 \|U\|^2+ \|V\|^2 + \|Z\|^2\right) + 2d(c_1 \|\nb U\|^2 + \|\nb V\|^2 + \|\nb Z\|^2) \\[2pt]
		&\, + \int_{\gw} (4U^4 + V^2 + rZ^2)\, dx \\
		\leq &\; 2 c(h) |\G(\theta_t \om)|^2 + 4 (c_1 b)^4 \|\G^h(\theta_t \om)\|_{L^4}^4 + \left[4\beta^2 +(c_1 a)^2 + \frac{1}{2} \right]\|\G^h(\theta_t \om)\|^2 + 2N |\gw|.
	\end{split}
	\eeq
Since $4U^4 \geq c_1 U^2 - \frac{c_1^2}{16}$, the inequality \eqref{snd16} implies that
	\beq \bl{snd17}
	\begin{split}
		&\, \frac{d}{dt} \left(c_1 \|U\|^2+ \|V\|^2 + \|Z\|^2\right)  \\[5pt] 
		&\,+ 2d(c_1 \|\nb U\|^2 + \|\nb V\|^2 + \|\nb Z\|^2)+ c_1 \|U\|^2 + \|V\|^2 + r\|Z\|^2 \\[6pt]
		\leq &\; 2 c(h) |\G(\theta_t \om)|^2 + 4 (c_1 b)^4 \|\G^h(\theta_t \om)\|_{L^4}^4 \\[3pt]
		&\, + 2 \left[2\beta^2 +(c_1 a)^2 + \frac{1}{4} \right]\|\G^h(\theta_t \om)\|^2 + 2N |\gw| + \frac{c_1^2}{16}\, |\gw|  \\
		\leq &\; \mathscr{C}(h)\, (|\G(\theta_t \om)|^2 + |\G(\theta_t \om)|^4) + F |\gw|,
	\end{split}
	\eeq
	for $t \geq \tau,\, \om \in \mathfrak{Q}$, where the constant $F = 2N + \frac{c_1^2}{16}$ and $\mathscr{C}(h) > 0$ is a constant depending on $h$.
	
	Let $\sigma = \text{min}\,\{1, r\}$. Gronwall inequality applied to the inequality from \eqref{snd17},
	\begin{equation} \bl{Gineq}
	\begin{split}
		\frac{d}{dt} &\, \left(c_1 \|\|U(t)\|^2 + \|V(t)\|^2 + \|Z(t)\|^2\right) + \sigma (c_1 \|U(t)\|^2 + \|V(t)\|^2 + \|Z(t)\|^2) \\[3pt]
		\leq &\; \mathscr{C}(h)\, (|\G(\theta_t \om)|^2 + |\G(\theta_t \om)|^4) + F |\gw|,
	\end{split}
	\end{equation}
shows that
	\beq \bl{snd18}
	\begin{split}
		c_1 &\, \|U(t)\|^2+ \|V(t)\|^2 + \|Z(t)\|^2 \leq e^{-\sigma(t - \tau)} (c_1 \|U_0\|^2+ \|V_0\|^2 + \|Z_0\|^2) \\
		+ &\, \int_{\tau}^{t} e^{-\sigma(t - s)} \left(\mathscr{C}(h)\, (|\G(\theta_s \om)|^2 + |\G(\theta_s \om)|^4) + F |\gw|\right) ds.
	\end{split}
	\eeq
It means that the weak solutions of the problem \eqref{sGq1} satisfy
	\beq \bl{snd19}
	\begin{split}
		&\, \|G(t, \theta_\tau \om; \tau,g_0)\|^2 = \|U(t)\|^2+  \|V(t)\|^2 + \|Z(t)\|^2  \\[5pt]
		&\, \leq \frac{\text{max}\,\{1, c_1\}}{\text{min}\,\{1, c_1\}}e^{-\sigma(t - \tau)} \|g_0 - \G^h(\theta_\tau \om)\|^2 \\
		&\, + \frac{1}{\text{min}\&,\{1, c_1\}} \int_{\tau}^{t} e^{-\sigma(t - s)}\, \left(\mathscr{C}(h)\, (|\G(\theta_s \om)|^2 + |\G(\theta_s \om)|^4) + F |\gw|\right) ds \\
		&\, \leq \frac{\text{max}\,\{1, c_1\}}{\text{min}\,\{1, c_1\}}e^{-\sigma(t - \tau)} \|g_0 - \G^h(\theta_\tau \om)\|^2 \\
		&\, + \frac{1}{\text{min}\,\{1, c_1\}} \int_{-\infty}^{t} e^{-\sigma(t - s)}\, \left(\mathscr{C}(h)\, (|\G(\theta_s \om)|^2 + |\G(\theta_s \om)|^4) + F |\gw|\right) ds 
	\end{split}
	\eeq
	for $t \geq \tau \in \mathbb{R},\, \om \in  \mathfrak{Q}$ and $g_0 \in H$.
	
	Since the Ornstein-Uhlenbeck process $\G(\theta_t \om)$ is tempered, the last integral in \eqref{snd19} is convergent. Therefore, the estimate \eqref{snd19} shows that the weak solution of the initial value problem \eqref{sUq1}--\eqref{inc4} will never blow up at any finite time $t \geq \tau$. The time interval of maximal existence of any weak solution is always $[\tau, \infty)$.
	\end{proof}
	
	\begin{lemma} \bl{lm2}
	There exists a random variable $R_0(\om)>0$ depending only on the parameters such that for any tempered random variable $\rho(\om)>0$ there exists a random variable $T(\rho,\om) > 0$ and the following statement holds\textup{:} For any $\tau \leq - T(\rho,\om), \, \om \in \mathfrak{Q}$, and any initial data $g_0=(u_0, v_0, z_0) \in H$ with $\|g_0\| \leq \rho(\theta_\tau \om)$, the weak solution $G(t, \theta_\tau \om; \tau, g_0)$ of the problem \eqref{sUq1}--\eqref{inc4} uniquely exists on $[\tau,\infty)$ and satisfies
		\beq \bl{snd20}
		\|G(0, \theta_\tau \om; \tau, g_0)\|^2 + \int_{-1}^{0} \|\nb G(s, \theta_\tau \om; \tau, g_0)\|^2\, ds \leq R_0(\om).
		\eeq
	\end{lemma}

	\begin{proof}
		Let $t = -1$. From the already shown inequality \eqref{snd19}, we get
		\beq \bl{snd21}
		\begin{split}
			&\|G(-1, \theta_\tau \om; \tau,g_0)\|^2 \leq \frac{\text{max}\,\{1, c_1\}}{\text{min}\,\{1, c_1\}}e^{\sigma(1 + \tau)} \|g_0 - \G^h(\theta_\tau \om)\|^2 \\
		 + &\, \frac{1}{\text{min}\,\{1, c_1\}} \int_{\infty}^{-1} e^{\sigma(1 + s)} \left(\mathscr{C}(h)\, (|\G(\theta_s \om)|^2 + |\G(\theta_s \om)|^4) + F |\mathfrak{Q}|\right) ds.
		\end{split}
		\eeq
		Thus for any given random variable $\rho(\om) > 0$ and for all $\om \in \mathfrak{Q}$, there exists a time $T(\rho,\om) > 1$ such that for any $\tau \leq - T(\rho,\om)$ we have 
		\begin{equation} \bl{snd22}
		\begin{split}
			&\, \frac{\text{max}\,\{1, c_1\}}{\text{min}\,\{1, c_1\}}e^{\sigma(1 + \tau)} \|g_0 - \G^h(\theta_\tau \om)\|^2 \\
			\leq &\; 2\, \frac{\text{max}\,\{1, c_1\}}{\text{min}\,\{1, c_1\}}e^{\sigma(1 + \tau)} \left(\rho^2(\om) + \|\G^h(\theta_\tau \om)\|^2 \right) \leq 1,
		\end{split}
		\end{equation}
		since $\G^h(\theta_t \om)$ is tempered.  Substituting the above inequality into \eqref{snd21}, we obtain 
		$$
			\|G(-1, \theta_\tau \om; \tau,g_0)\|^2  \leq r_0(\om),
		$$
		where 
		\beq \bl{snd23}
			r_0(\om) = 1 + \frac{1}{\text{min}\,\{1, c_1\}} \int_{\infty}^{-1} e^{\sigma(1 + s)} \left(\mathscr{C}(h) (|\G(\theta_s \om)|^2 + |\G(\theta_s \om)|^4) + F |\gw|\right) ds.
		\eeq
		For $t \in [-1, \infty)$, integrate the inequality \eqref{snd17} over $[-1, t]$ to get	
		\beq \bl{snd24}
		\begin{split}
			&c_1 \|U(t)\|^2 + \|V(t)\|^2 +\|Z(t)\|^2 - (c_1 \|U(-1)\|^2 + \|V(-1)\|^2 +\|Z(-1)\|^2) \\[3pt]
			&\, + 2d \int_{-1}^{t}(c_1 \|\nb U(s)\|^2 + \|\nb V(s)\|^2 + \|\nb Z(s)\|^2)\, ds \\
			&\,+ \sigma \int_{-1}^{t}(c_1 \|U(s)\|^2 + \|V(s)\|^2 + \|Z(s)\|^2)\, ds  \\
			&\leq  \int_{-1}^{t} \left[ \mathscr{C}(h) \left(|\G (\theta_s \om)|^2 + |\G (\theta_s \om)|^4\right) + F |\gw|\right] ds.
		\end{split}
		\eeq
		Thus for $t \in [-1, 0]$ we have
		\beq \bl{snd25}
		\begin{split}
			& \|G(t, \theta_\tau \om;  \tau,g_0)\|^2 + 2d \int_{-1}^{t} \|\nb G(s, \theta_\tau \om; \tau,g_0)\|^2\, ds\\[3pt]
			\leq &\, \frac{\text{max}\{c_1, 1\}}{\text{min}\{c_1, 1\}} \|G(-1, \theta_\tau \om; \tau, g_0)\|^2 + \int_{-1}^{t} \left[\mathscr{C}(\ga)\, \left(|\G(\theta_s \om)|^2 + |\G(\theta_s \om)|^4\right) + F |\gw|\right] ds.
		\end{split}
		\eeq
Let $t = 0$ in \eqref{snd25} and we see that the claim \eqref{snd20} is proved:
		\beq \bl{snd26}
			\|G(0, \theta_\tau \om; \tau,g_0)\|^2 +  \int_{-1}^{0} \|\nb G(s, \theta_\tau \om; \tau, g_0)\|^2\, ds \leq R_0(\om),
		\eeq
		where
		\beq \bl{snd27}
		\begin{split}
			&\, R_0(\om) = \frac{1}{\text{min}\{1, 2d\} \text{min}\{c_1, 1\}}\\[3pt]
			&\, \times \left\{\text{max}\{c_1, 1\} r_0(\om) + \int_{-1}^{0} \left[\mathscr{C}(h)\, \left(|\G(\theta_s \om)|^2 + |\G(\theta_s \om)|^4\right) + F |\gw|\right]\, ds \right\}
		\end{split}
		\eeq
Note that both $r_0(\om)$ and $R_0(\om) $ are random variables independent of any initial data. The proof is compldeted.
		\end{proof}
		
The two lemmas that we have shown expose the longtime dissipativity for pullback solution trajectories of the stochastic Hindmarsh-Rose cocycle to be defined in the next subsection.
	
\subsection{\textbf{Hindmarsh-Rose Cocycle and Absorbing Property}}
		Now define a concept of stochastic semiflow, which is related to the concept of cocycle in the theory of random dynamical systems.
		
		\begin{definition} \bl{SSM}
			Let $(\mathfrak{Q}, \mathcal{F}, P, \{\theta_t\}_{t\in \mathbb{R}})$ be a metric dynamical system. A family of mappings $S(t,\tau,\om): X \to X$ for $t \geq \tau \in \mathbb{R}$ and $\om \in \mathfrak{Q}$ is called a \emph{stochastic semiflow} on a Banach space $X$, if it satisfies the properties:
			
			(i)\, $S(t,s, \om) S(s,\tau,\om) = S(t,\tau,\om)$, for all $\tau \leq s \leq t$ and $\om \in \mathfrak{Q}$.
			
			(ii) $S(t, \tau, \om) = S(t-\tau, 0, \theta_\tau \om)$, for all $\tau \leq t$ and $\om\in \mathfrak{Q}$.
			
			(iii) The mapping $S(t,\tau,\om) x$ is measurable in $(t,\tau,\om)$ and continuous in $x \in X$.
		\end{definition}
		We can define the stochastic semiflow associated with the random PDE \eqref{sUq1}--\eqref{sZq1} and then the cocycle $\Phi: \mathbb{R}^+ \times \mathfrak{Q} \times H \rightarrow H$ over $(\mathfrak{Q}, \mathcal{F}, P, \{\theta_t\}_{t\in \mathbb{R}})$ for the stochastic Hindmarsh-Rose equations. 
		
		For all $t \geq \tau \in \mathbb{R}$ and $\om \in \mathfrak{Q}$, define $S(t,\tau,\om): H \to H$ to be
		\beq \bl{snd28}
			S(t,\tau, \om) g_0 = (u,v,z)(t,\om;\tau,g_0) = G(t,\om;\tau,g_0) + \G^h(\theta_t \om).
		\eeq
Then define the mapping $\Phi: \mathbb{R}^+ \times \mathfrak{Q} \times H \to H$ to be
		\beq \bl{snd29}
		\Phi (t, \, \theta_\tau \om, \, g_0) = S(t + \tau, \,\tau, \,\theta_\tau \om)\, g_0
		\eeq
which implies that
		\beq \bl{snd30}
		\Phi (t,\, \om, \, g_0) = S(t,\, 0, \,\om)\,g_0 = G(t,\om; \, 0, g_0) + \G^h(\theta_t \om).
		\eeq
		
\begin{lemma}
		The mapping $\Phi: \mathbb{R}^+ \times \mathfrak{Q} \times H \to H$ defined by \eqref{snd29} is a cocycle on the Hilbert space $H$ over the canonical metric dynamical system $(\mathfrak{Q}, \mathcal{F}, P, \{{\theta_t}\}_{t \in \mathbb{R}})$. It holds that
		\beq \bl{snd31}
		\Phi (t, \theta_{-t} \om, g_0) = G(0, \theta_{-t} \om; -t, g_0) + \G^h(\om)
		\eeq
		for any $g_0 \in H, t \geq 0$ and $\om \in \mathfrak{Q}$. This random dynamical system $\Phi$ is called the stochastic Hindmarsh-Rose cocycle. 
\end{lemma}

\begin{proof}
We need to check the cocyle property of the mapping $\Phi$:
\beq \bl{snd33}
\Phi (t + s, \om, g_0) = \Phi (t, \theta_s \, \om, \Phi (s, \om, g_0)), \quad  t \geq 0, \, s \geq 0, \, \om \in \mathfrak{Q},
\eeq
Note that, \eqref{snd1} and \eqref{snd3} imply that for any $\om \in \mathfrak{Q}$,
$$
	\G^h(\theta_{s} \om)(t) = \G^h(\om)(t + s), \quad t, s \in \mathbb{R},
$$
and
$$
	 \G^h(\theta_{s} \om) = \G^h(\theta_{s} \om)(0) = \G^h(\om)(s) .
$$
According to \eqref{snd30},
$$
	\Phi (t+s, \om, g_0) = G(t+s,\om;0,g_0) + \G^h(\theta_{t+s}\om).
$$
On the other hand,
\begin{align*} 
&\Phi (t, \theta_s \om, \Phi (s, \om, g_0)) = G(t, \theta_s \om; 0, \Phi (s, \om, g_0)) + \G^h(\theta_t\, \theta_s\, \om)\\
= &\, S(t, 0, \zsw)\, (G(s, \om; 0, g_0) + \G^h(\theta_s \om))  \quad (\tup{by}\, \eqref{snd31}) \\[4pt]
= &\, S(t, 0, \zsw)\, S(s, 0, \om) g_0  = S (t+s-s, 0, \zsw)\, S(s, 0, \om) g_0  \\[4pt]
= &\, S(t+s, s, \om)\, S(s, 0, \om) g_0  \qquad (\tup{by the second condition of Definition \ref{SSM}}) \\
= &\, S(t+s, 0, \om)\, g_0 = G (t + s, \om; 0, g_0) + \G^h (\theta_{t + s} \om). 
\end{align*} 
Therefore, the cocycle property \eqref{snd33} of the mapping $\Phi$ is proved by comparison of the above two equalities. Moreover, by definition we have
$$
\Phi (t, \ctw, g_0) = S(0, -t,  \theta_{-t} \om)g_0 =  G (0, \theta_{-t} \om; -t, g_0) + \G^h (\theta_t(\theta_{-t} \om)).
$$
Thus the equality \eqref{snd31} is valid. 
\end{proof}

\begin{theorem} \bl{Thm2}
	There exists a pullback absorbing set in the space $H$ with respect to the universe $\mathscr{D}_H$ for the stochastic Hindmarsh-Rose cocycle $\Phi$, which is the bounded ball
	\beq \bl{snd34} 
	K(\om) = B_H(0, R_H(\om)) = \{\xi \in H: \|\xi\| \leq R_H (\om)\}
	\eeq
	where $R_H(\om) = \sqrt{R_0 (\om) + \|\G^h(\om)\|^2}$ and $R_0(\om)$ is given in Lemma \ref{lm2} by \eqref{snd27}.
\end{theorem}

\begin{proof}
	For any bounded random ball $D(\om) = B_H(0, \rho(\om)) \in \mathscr{D}_H$, which is centered at the origin with the radius $\rho (\om)$ in $H$, and for any initial state $g_0 \in D(\theta_{-t}\om)$, by Definition \ref{sc3} and the definition of the universe $\mathscr{D}_H$, we have
	\beq \bl{snd35}
	\lim_{t \to -\infty} e^{-\ve t} \rho(\theta_{-t} \om) = 0,\quad \text{for any const} \; \ve > 0.
	\eeq
	From \eqref{snd21}, for any $t > 1$ we have
	\begin{equation*}
	\begin{split}
	&\, \interleave G(-1, \theta_{-t} \om; -t, D(\theta_{-t}\om))\interleave  = \sup_{g_0 \in D (\theta_{-t}\om)} \|G(-1, \theta_{-t} \om; -t, g_0)\| \\[3pt]
	\leq &\;  2\, \frac{\text{max}\,\{1, c_1\}}{\text{min}\,\{1, c_1\}}e^{\sigma(1 - t)} \left(\rho^2 (\theta_{-t} \om) + \|\G^h(\theta_{-t} \om)\|^2 \right) \\
	&\, + \frac{1}{\text{min}\,\{1, c_1\}} \int_{\infty}^{-1} e^{\sigma(1 + s)}\left(\mathscr{C}(h)\, (|\G(\theta_s \om)|^2 + |\G(\theta_s \om)|^4) + F |\gw|\right) ds.
	\end{split}
	\end{equation*}
	Since $\G^h(\theta_t \om)$ and $\G(\theta_t \om)$ are tempered random variables, there exists a time $T_D(\om) > 1$ such that for any $t > T_D(\om)$ and $\om \in \mathfrak{Q}$ we have 
	\begin{equation} \bl{lq1}
	2\, \frac{\text{max}\,\{1, c_1\}}{\text{min}\,\{1, c_1\}}e^{\sigma(1 - t)} \left(\rho^2(\theta_{-t} \om) + \|\G^h(\theta_{-t} \om)\|^2 \right) \leq 1.
	\end{equation}
Thus 
	$$
	\sup_{g_0 \in D(\theta_{-t}\om)} \|G(-1, \theta_{-t} \om; -t, g_0)\| \leq r_0(\om),\quad \text{for all}\; t\geq T_D(\om),
	$$
	where $r_0(\om)$ is given in \eqref{snd23}.By \eqref{snd31} and the inequalities \eqref{snd25}-\eqref{snd26} in Lemma \ref{lm2}, the above inequality implies that
	\begin{equation*}
	\begin{split}
	&\, \interleave \Phi(t, \theta_{-t} \om, D(\theta_{-t}\om))\interleave =  \interleave G(0, \theta_{-t} \om; -t, D(\theta_{-t}\om)) + \G^h(\om)\interleave  \\[3pt]
	=&\,  \sup_{g_0 \in D(\theta_{-t}\om)} \|G(0, \theta_{-t}\om; -t, g_0) + \G^h(\om)\| \leq \sqrt{R_0(\om) + \|\G^h(\om)\|^2} =: R_H(\om),
	\end{split}
	\end{equation*}
	for $t \geq T_D(\om),\,  \om \in \mathfrak{Q}$, where $R_0(\om)$ is given in \eqref{snd27}.
	It shows that the bounded ball $K(\om) = B_H(0, R_H(\om))$ in \eqref{snd34} is a pullback absorbing set for Hindmarsh-Rose random dynamical system $\Phi$. 
\end{proof}

\section{\textbf{The Existence of Random Attractor}}

In this section, we shall prove that the stochastic Hindmarsh-Rose cocycle $\Phi$ is pullback asymptotically compact on $H$ through the following theorem. Then the main result on the existence of a random attractor for this random dynamical system is established.

\begin{theorem} \bl{Thm3}
	For the Hindmarsh-Rose random dynamical system $\Phi$ with the assumption that space dimension $n = \dim \, (\gw) \leq 2$, there exists a random variable $R_E(\om) > 0$ independent of any initial time and initial state with the property that for any bounded random set $D \in \mathscr{D}_H$ there is a finite time $T(D,\om) > 0$ such that 
	\beq \bl{snd36}
	\interleave \Phi(t, \theta_{-t}\,\om,D(\theta_{-t} \om)\interleave_E = \sup_{g_0 \in D(\theta_{-t} \om)} \|\Phi(t, \theta_{-t}\,\om, g_0\|_E \leq R_E(\om).
	\eeq
for all $t  > T(D,\om)$.
\end{theorem}

\begin{proof}
	We can just consider any bounded ball $D = B_H (0, \rho (\om)) \in \mathscr{D}_H$ in this proof.
	
	Step 1. Respectively take the $L^2$ inner-products $\inpt{\eqref{sUq1}, -\gd U(t)}, \inpt{\eqref{sVq1}, -\gd V(t)}$ and $\inpt{\eqref{sZq1}, -\gd Z(t)}$. Sum up the resulting equalities. For any $t > \tau \in \mathbb{R}$, we have
	\beq \bl{snd37}
	\begin{split}
		& \frac{1}{2} \frac{d}{dt} (\|\nb U\|^2 + \|\nb V\|^2 + \|\nb Z\|^2) + d_1 \|\gd U\|^2 + d_2 \|\gd V\|^2 + d_3 \|\gd Z\|^2\\[3pt]
		= &\, -\int_{\gw} J \gd U dx - \int_{\gw} \alpha \gd V dx + \int_{\gw} qc \gd Z dx  \\  
		&\, - \int_{\gw}  \gd U \left[d_1 \gd h_1 \G_1 (\theta_t \om_1) + \gk \G_1^h (\theta_t \om_1)\right] dx \\
		&\, -  \int_{\gw}  \gd V \left[d_2  \gd h_2 \G_2(\theta_t \om_2) + \gk \G_2^h(\theta_t \om_2)\right] dx \\
		&\, - \int_{\gw}  \gd Z \left[d_3  \gd h_3 \G_3(\theta_t \om_3) + \gk \G_3^h(\theta_t \om_3)\right] dx \\
		&\, + \int_{\gw}  [\gd U (b u^3 -a u^2 - v + z) + \gd V (\beta u^2 + v) + \gd Z (rz -qu)]\,dx 
	\end{split}
	\eeq
The key last integral on right-hand side of \eqref{snd37} can be written as
	\beq \bl{snd39}
	\begin{split}
		&\, \int_{\gw}  \left[\gd U (b u^3 -a u^2 - v + z) + \gd V (\beta u^2 + v) + \gd Z (rz -qu)\right] dx\\
		= &\, \int_{\gw} \left[(bu^3 - au^2 -v + z) \gd u + (\beta u^2 + v) \gd v + (rz -qu) \gd z\right] dx\\
		&\,+ \int_{\gw} [-(bu^3 - au^2 -v + z) \gd \G_1^h(\theta_t \om_1) - (\beta u^2 + v) \gd \G_2^h(\theta_t \om_2)  \\[4pt]
		&\, + (rz -qu) \gd \G_3^h(\theta_t \om_3)] \, dx.
	\end{split}
	\eeq 
The first integral on the right-hand side of \eqref{snd39} is estimated as follows.
	\beq \bl{snd40}
	\begin{split}
		&\, \int_{\gw} \left[(bu^3 - au^2 -v + z) \gd u + (\beta u^2 + v) \gd v + (rz -qu) \gd z\right] dx\\[2pt]
		= &\, -3b \int_{\gw} u^2 |\nb u|^2 \,dx + 2a \int_{\gw} u |\nb u|^2 \,dx + \int_{\gw} \nb u \cdot \nb v \,dx - \int_{\gw} \nb z \cdot \nb u \,dx\\
		&\, - 2 \beta \int_{\gw} u \nb u \cdot \nb v \,dx - \int_{\gw} |\nb v|^2 \,dx - r\int_{\gw} |\nb z|^2 \,dx + q \int_{\gw} \nb u \cdot \nb z \,dx\\
		\leq &\, \frac{-3b}{4} \|\nb (u^2)\|^2 + 2 a \|u\|_{L^{\infty}} \|\nb u\|^2 + \|\nb u\|^2 + \|\nb v\|^2 + \|\nb z\|^2 + \|\nb u\|^2 \\[10pt]
		&\, + \beta \|u\|_{L^{\infty}} (\|\nb u\|^2 + \|\nb v\|^2) - \|\nb v\|^2 - r \|\nb z\|^2 + q (\|\nb u\|^2 + \|\nb z\|^2)\\[12pt]
		\leq &\, 2 C\, \text{max}\,\{2a, \beta\} \|u\|_{H^1} (\|\nb u\|^2 + \|\nb v\|^2) \\[12pt]
		&\, + \max \,\{2,q\} (\|\nb u\|^2 + \|\nb v\|^2 + \|\nb z\|^2)  \\[12pt]
		\leq &\, \tilde{C}\, \text{max}\,\{2a, \beta\} (\|\nb u\|^3 + \|\nb u\| \|\nb v\|^2) \\[12pt]
		&\, + \max \,\{2,q\} (\|\nb u\|^2 + \|\nb v\|^2 + \|\nb z\|^2)  \\[8pt]
		\leq &\, \frac{C_1}{4}\left(\|\nb G\|^3 + \|\nb \G^h(\theta_t \om)\|^3 + \|\nb G\|^2 + \|\nb \G^h(\theta_t \om)\|^2 \right) \\[8pt]
		\leq &\, \frac{C_1}{4}\left(\frac{3}{4} \|\nb G\|^4 + \frac{1}{4} + \frac{3}{4} \|\nb \G^h(\theta_t \om)\|^4 +  \frac{1}{4} + \frac{1}{2} \|\nb G\|^4 + \frac{1}{2} \|\nb \G^h(\theta_t \om)\|^4\right)  \\[5pt]
		=&\, \frac{C_1}{4} \left(\frac{5}{4} \|\nb G\|^4 + \frac{5}{4} \|\nb \G^h(\theta_t \om)\|^4 + \frac{1}{2}\right) \leq \frac{C_1}{2} \left(\|\nb G\|^4 + \|\nb \G^h(\theta_t \om)\|^4 + 1\right),
	\end{split}
	\eeq
	where $C_1 > 0$ is constant and we have used the Young's inequality and \eqref{snd28}. For the second step of the chain inequalities in \eqref{snd40}, the Sobolev embedding $H^1(\gw) \hookrightarrow L^{\infty}(\gw)$ under the assumption $\text{dim}(\gw) \leq 2$ so that $\|u\|_{L^{\infty}} \leq C \|u\|_{H^1}$ is used to deal with the integral term
	$$
	- 2 \beta \int_{\gw} u \nb u \cdot \nb v \,dx.
	$$ 
	
	Next we estimate the second integral in \eqref{snd39}: 
	
	\beq \bl{snd41}
	\begin{split}
		\int_{\gw} &\, [(au^2 - bu^3 - v + z) \gd \G_1^h(\theta_t \om_1) - (\beta u^2 + v) \gd \G_2^h(\theta_t \om_2)  \\[3pt]
		 &\, + (rz -qu) \gd \G_3^h(\theta_t \om_3)] \, dx \\[6pt]
		= &\, \int_{\gw} [ u^2 \left(a \gd \G_1^h(\theta_t \om_1) - \beta \gd \G_2^h(\theta_t \om_2)\right) + bu^3 \gd \G_1^h(\theta_t \om_1) \\[6pt]
		&\, - v(\gd \G_1^h(\theta_t \om_1) + \gd \G_2^h(\theta_t \om_2))  \\[10pt]
		&\, +  z(\gd \G_1^h(\theta_t \om_1) - r \gd \G_3^h(\theta_t \om)_3) + qu \gd \G_3^h(\theta_t \om_3)] \, dx \\[6pt]
		\leq &\, \int_{\gw} \left[ \frac{u^4}{4} + \left(a \gd \G_1^h(\theta_t \om_1) - \beta \gd \G_2^h(\theta_t \om_2)\right)^2 + \frac{3}{4}u^4 + \frac{b^4}{4} (\gd \G_1^h(\theta_t \om_1))^4 + \frac{v^2}{2} \right.\\
		& \left. + \frac{1}{2}(\gd \G_1^h(\theta_t \om_1) + \gd \G_2^h(\theta_t \om_2))^2 + \frac{z^2}{2} + \frac{1}{2}(\gd \G_1^h(\theta_t \om_1) - r \gd \G_3^h(\theta_t \om_3))^2 \right.\\[5pt]
		& \left. + \frac{u^2}{2} + \frac{q^2}{2} (\G_3^h(\theta_t \om_3))^2 \right] dx.
	\end{split}
	\eeq
	
	Step 2. We further treat the integral in the last step of \eqref{snd41}, which is decomposed into the following two parts. The first part is	
	\beq \bl{snd42}
	\begin{split}
		&\, \int_{\gw} \left(u^4 + \frac{u^2}{2} + \frac{v^2}{2} + \frac{z^2}{2}\right) dx = \int_{\gw} \left[\left(U + \G_1^h(\theta_t \om_1) \right)^4 + \frac{1}{2} \left(U + \G_1^h(\theta_t \om_1) \right)^2 \right.\\
		&\left. +\,   \frac{1}{2} \left(V + \G_2^h(\theta_t \om_2) \right)^2 + \frac{1}{2} \left(Z + \G_3^h(\theta_t \om_3) \right)^2\right] dx\\
		\leq & \int_{\gw} \left[8 \left(U^4 + (\G_1^h(\theta_t \om_1))^4\right) + U^2 + (\G_1^h(\theta_t \om_1))^2  \right. \\[5pt]
		&\left. +\, V^2 + (\G_2^h(\theta_t \om_2))^2 + Z^2 + (\G_3^h(\theta_t \om_3))^2\right] dx\\[5pt]
		\leq & \int_{\gw} (8U^4(t) + U^2(t) + V^2(t) + Z^2(t))\, dx + 8 \|\G^h(\theta_t \om)\|_{L^4}^4 + \|\G^h(\theta_t \om)\|^2.
	\end{split}
	\eeq 
According to \eqref{snd19} and $D = B_H (0, \rho(\om))$, for $t \geq \tau \in \mathbb{R}$,
\begin{equation*}
	\begin{split}
		 \|U(t)&\,\|^2 +  \|V(t)\|^2 + \|Z(t)\|^2  \leq \frac{\text{max}\,\{1, c_1\}}{\text{min}\,\{1, c_1\}} \,2e^{-\sigma(t - \tau)} (\rho^2 (\theta_\tau \om) + \| \G^h(\theta_\tau \om)\|^2)  \\
		&+ \frac{1}{\text{min}\,\{1, c_1\}} \int_{-\infty}^{t} e^{-\sigma(t - s)}\, \left(\mathscr{C}(h)\, (|\G(\theta_s \om)|^2 + |\G(\theta_s \om)|^4) + F |\gw|\right) ds. 
	\end{split}
\end{equation*}
The tempered property of $\rho^2 (\theta_\tau \om) + \| \G^h(\theta_\tau \om)\|^2$ implies that there is a sufficiently large random variable $T (D, \om) > 6$ such that if $\tau < - \,T(D, \om)$, then it holds that $\|U(t)\|^2 +  \|V(t)\|^2 + \|Z(t)\|^2  \leq Q_1 (\om)$ for any $t \in [\tau/2, 0]$,
where
\beq \bl{Q1}
	\begin{split}
	Q_1 (\om) = 1 +  \frac{1}{\min \{1, c_1\}} \int_{-\infty}^{0} e^{\sigma s} \left(\mathscr{C}(h)\, (|\G(\theta_s \om)|^2 + |\G(\theta_s \om)|^4) + F |\gw|\right) ds.
	\end{split}
\eeq
By the embedding $H^1(\gw) \hookrightarrow L^4 (\gw)$, there is a positive constant $\eta > 0$  such that $\|U\|^4_{L^4} \leq \eta (\|U\|^2 + \|\nb U\|^2)^2 \leq 2 \eta (\|U\|^4 +\|\nb U \|^4)$.  It follows from \eqref{snd42} that
	\beq \bl{snd43}
	\begin{split}
		&\, \int_{\gw} \left(u^4(t) + \frac{u^2(t)}{2} + \frac{v^2(t)}{2} + \frac{z^2(t)}{2}\right)\, dx  \\[3pt]
		\leq &\, 16\, \eta\, \|U(t)\|^4  + 16\, \eta\, \|\nb U(t)\|^4 + Q_1 (\om) + 8 \|\G^h(\theta_t \om)\|_{L^4}^4 + \|\G^h(\theta_t \om)\|^2 \\[8pt]
		\leq &\, 16\, \eta\, Q_1^2 (\om) + 16\, \eta\, \|\nb G(t)\|^4 + Q_1 (\om) + 8 \|\G^h(\theta_t \om)\|_{L^4}^4 + \|\G^h(\theta_t \om)\|^2 \\
	\end{split}
	\eeq
provided that $t \in [\tau/2, 0]$ and $\tau < - T(D, \om)$.

	For the second part (the rest part) in the last integral of \eqref{snd41}, we have 
	\beq \bl{snd44}
	\begin{split}
		&\, \int_{\gw} \left[\left(a \gd \G_1^h(\theta_t \om_1) - \beta \gd \G_2^h(\theta_t \om_2)\right)^2 + \frac{b^4}{4} (\gd \G_1^h(\theta_t \om_1))^4 \right. \\
		& \quad + \left. \frac{1}{2}\left(\gd \G_1^h(\theta_t \om_1) + \gd \G_2^h(\theta_t \om_2)\right)^2 \right.\\
		& \quad + \left. \frac{1}{2}\left(\gd \G_1^h(\theta_t \om_1) - r \gd \G_3^h(\theta_t \om_3)\right)^2 + \frac{q^2}{2} (\G_3^h(\theta_t \om_3))^2 \right] dx\\
		\leq &\, \int_{\gw} \left[2a^2 (\gd h_1(x))^2 (\G_1(\theta_t \om_1))^2 + 2 \beta^2 (\gd h_2(x))^2 (\G_2(\theta_t \om_2))^2 \right.\\[3pt]
		& \quad + \left. \frac{1}{2}b^4 (\gd h_1(x))^4 (\G_1(\theta_t \om_1))^4 + 2 (\gd h_1(x))^2 (\G_1(\theta_t \om_1))^2 \right.\\[8pt]
		& \quad + \left. (\gd h_2(x))^2 (\G_2(\theta_t \om_2))^2 + (r^2 + q^2) (\gd h_3(x))^2 (\G_3(\theta_t \om_3))^2 \right]dx\\[8pt]
		= &\, \int_{\gw} \left[2 (a^2  + 1)(\gd h_1)^2 (\G_1(\theta_t \om_1))^2 + (2 \beta^2  + 1)(\gd h_2)^2 (\G_2(\theta_t \om_2))^2 \right.\\
		& \quad + \left. (r^2 + q^2)(\gd h_3)^2 (\G_3(\theta_t \om_3))^2 + \frac{1}{2}b^4 (\gd h_1)^4 (\G_1(\theta_t \om_1))^4 \right] dx \\[5pt]
		\leq &\,\; |\G(\theta_t \om)|^2 \left[2 (a^2  + 1)\|\gd h_1\|^2 +  (2 \beta^2  + 1)\|\gd h_2\|^2\right]  \\[5pt]
		&\quad + |\G(\theta_t \om)|^2 \left[ (r^2 + q^2)\|\gd h_3\|^2 \right] + \frac{1}{2}\, b^4 |\G(\theta_t \om)|^4 \|\gd h_1\|_{L^4}^4.	
		\end{split}
		\eeq
In \eqref{snd44}, the assumption that $\{h_i(x): i=1, 2, 3\} \subset W^{2,4}(\gw)$ specified in Section 1 is used.
		
		Step 3. Assemble the estimates \eqref{snd43} and \eqref{snd44} of the two parts in \eqref{snd41}. Then we have proved that 
		\beq \bl{snd46}
		\begin{split}
		 \int_{\gw} [(au^2 - bu^3 - v + z) \gd \G_1^h(\theta_t \om) &\, - (\beta u^2 + v) \gd \G_2^h(\theta_t \om) + (rz -qu) \gd \G_3^h(\theta_t \om)] \, dx\\[3pt]
			\leq&\, 16\, \eta\, \|\nb G\|^4 + Q_2(t, \om),
		\end{split}
		\eeq
where 
\beq \bl{Q2}
	\begin{split}
	Q_2 (t, \om) &= 16\, \eta\, Q_1^2 (\om) + Q_1 (\om) + 8 \|\G^h(\theta_t \om)\|_{L^4}^4 + \|\G^h(\theta_t \om)\|^2 \\[5pt]
	&\, + |\G(\theta_t \om)|^2 \left[2 (a^2  + 1)\|\gd h_1\|^2 +  (2 \beta^2  + 1)\|\gd h_2\|^2\right]  \\
  	&\, + |\G(\theta_t \om)|^2 \left[ (r^2 + q^2)\|\gd h_3\|^2 \right] + \frac{1}{2}\, b^4 |\G(\theta_t \om)|^4 \|\gd h_1\|_{L^4}^4.
	\end{split}
\eeq
In turn, substitute the inequalities \eqref{snd40} and \eqref{snd46} into \eqref{snd39}, we get
		\beq \bl{snd47}
		\begin{split}
			&\, \int_{\gw}  \left[(b u^3 -a u^2 - v + z) \gd U  + (\beta u^2 + v) \gd V  + (rz -qu) \gd Z  \right] dx\\[3pt]
			\leq &\, \left(\frac{C_1}{2} + 16\, \eta \right) \|\nb G\|^4 + \frac{C_1}{2} \left(\|\nb \G^h(\theta_t \om)\|^4 + 1\right) + Q_2( t, \om).
		\end{split}
		\eeq
		
	Besides, by the Gauss Divergence theorem and the homogeneous Neumann boundary condition, in \eqref{snd37} we have 
		$$
		\int_{\gw} J\, \gd U dx = \int_{\gw} \alpha \,\gd V dx = \int_{\gw} q\,c \,\gd Z dx = 0.
		$$
Moreover, the three middle terms in \eqref{snd37} satisfies the estimates
		\beq \bl{snd38}
		\begin{split}
			&\, - \int_{\gw}  \gd U \left[d_1  \gd h_1 \G_1 (\theta_t \om_1) + \gk \G_1^h (\theta_t \om_1)\right] dx \\[3pt]
			&\, -  \int_{\gw}  \gd V \left[d_2  \gd h_2 \G_2(\theta_t \om_2) + \gk \G_2^h(\theta_t \om_2)\right] dx\\
			&\, - \int_{\gw}  \gd Z \left[d_3  \gd h_3 \G_3(\theta_t \om_3) + \gk \G_3^h(\theta_t \om_3))\right] dx \\
			\leq &\;  \frac{1}{2} \left(d_1 \|\gd U\|^2 + d_2 \|\gd V\|^2 + d_3 \|\gd Z\|^2\right) + \frac{1}{2} \,C_2 (h) \,|\G(\theta_t \om)|^2,
		\end{split}
		\eeq
		where $C_2 (h) > 0$ is a constant only depending on the functions $\{h_1, h_2, h_3\}$.
		
	Finally, we substitute \eqref{snd47} and \eqref{snd38} into the inequality \eqref{snd37}. It follows that
		\beq \bl{snd48}
		\begin{split}
			&\, \frac{d}{dt} \|\nb G(t)\|^2 + d_1 \|\gd U(t)\|^2 + d_2 \|\gd V(t)\|^2 + d_3 \|\gd Z(t)\|^2\\[3pt]
			\leq (C_1 + 32\, \eta) &\,\|\nb G(t)\|^4 + C_1 \left(\|\nb \G^h(\theta_t \om)\|^4 + 1\right) + 2 Q_2(t, \om) + C_2(h) |\G(\theta_t \om)|^2.
		\end{split}
		\eeq

	Step 4. In the final step of this proof, we apply the uniform Gronwall inequality \cite{SY} to the following differential inequality reduced from \eqref{snd48},
		\beq \bl{snd49}
		\begin{split}
			 \frac{d}{dt} \|\nb G(t)\|^2 \leq&\, \left(C_1 + 32\, \eta \right) \|\nb G\|^4 + C_1 \left(\|\nb \G^h(\theta_t \om)\|^4 + 1\right)\\[3pt]
		 + &\, 2 Q_2(t, \om) + C_2(h) |\G(\theta_t \om)|^2,
		\end{split}
		\eeq
		which can be written in the form
		\beq \bl{snd50}
		\frac{d \zeta}{dt} \leq \gl \, \zeta + h,\quad \text{for} \;\, t \in [\tau/2, 0], \, \tau < - T(D, \om),
		\eeq
		where $T(D, \om) > 6$ as specified before \eqref{Q1}, and
		\begin{gather*}
		\zeta(t) =  \|\nb G(t)\|^2, \\[5pt]
		\gl (t) = (C_1 + 32\, \eta) \|\nb G(t)\|^2,\\[5pt]
		h(t) = C_1 \left(\|\nb \G^h(\theta_t \om)\|^4 + 1\right) + 2 Q_2(t, \om) + C_2(h) |\G(\theta_t \om)|^2.
		\end{gather*}
		
		To estimate the functions $\zeta (t)$ and $\gl (t)$, we integrate of the inequality \eqref{snd17} over the time interval $[t-1, t] \subset [\tau/2, 0]$ to get
\begin{equation} \bl{intG}
	\begin{split}
	& 2d \int_{t-1}^t \|\nb G(s, \theta_\tau \om; \,\tau, g_0 - \G^h (\theta_\tau \om))\|^2 ds \\
	\leq &\, \frac{\max \{1, c_1\}}{\min \{1, c_1\}} \|G(t-1, \theta_\tau \om; \tau, g_0 - \G^h (\theta_\tau \om))\|^2 \\
	+ &\, \int_{t-1}^t \left[ \mathscr{C}(h) \left(|\G(\theta_s \om)|^2 + |\G(\theta_s \om)|^4\right) + F |\gw|\right] ds
	\end{split}
\end{equation} 
It has been shown in Step 2 that 
\beq \bl{Gt1}
	 \|G(t-1, \theta_\tau \om; \, \tau, g_0 - \G^h (\theta_\tau \om))\|^2 \leq Q_1 (\om)
\eeq
and $Q_1 (\om)$ is given in \eqref{Q1}. It follows from \eqref{intG} and \eqref{Gt1} that
		\beq \bl{snd51}
		\int_{t-1}^{t} \zeta(s)ds = \int_{t-1}^{t} \|\nb G(s, \theta_\tau \om;\, \tau,g_0 - \G^h(\theta_\tau \om))\|^2 ds \leq R_1 (\om),
		\eeq
for any $g_0 \in D(\theta_\tau \om), \, t \in [-2, 0] \subset [\tau/2 + 1, 0], \, \tau < -T(D,\om)$, where
		\beq \bl{snd52}
		\begin{split}
		R_1 (\om) &\,= \frac{1}{2d}  \left\{\frac{\text{max}\{c_1, 1\}}{\text{min}\{c_1, 1\}} Q_1(\om) \right.\\[3pt]
		 &\, \left.+ \int_{-3}^{0} \left[\mathscr{C}(h)\left(|\G(\theta_s \om)|^2 + |\G(\theta_s \om)|^4\right) + F |\gw|\right] ds \right\}.
		\end{split}
		\eeq
Then in the same way, for any $g_0 \in D(\theta_\tau \om), \, t \in [-2, 0], \, \tau < - T(D,\om)$, we have
		\beq \bl{snd53}
			\int_{t-1}^{t} \gl (s)\, ds \leq (C_2 + 32\, \eta) R_1 (\om).
		\eeq
Moreover, for $t \in [-2, 0], \, \tau < -T(D,\om)$, we have %
		\beq \bl{snd54}
		\begin{split}
			&\int_{t-1}^{t} h(s)\, ds  \\
			\leq &\, \int_{t-1}^{t} \left[C_1 \left(\|\nb \G^h(\theta_s \om)\|^4 + 1\right) + 2 Q_2(s, \om) + C_2(h) |\G(\theta_s \om)|^2 \right] ds\\
			\leq &\, \int_{-3}^{0} \left[C_1 \left(\|\nb \G^h(\theta_s \om)\|^4 + 1\right) + 2 Q_2(s, \om) + C_2(h) |\G(\theta_s \om)|^2 \right] ds.
		\end{split}
		\eeq
Therefore, for any $t \in [-2, 0], \tau \leq -T(D,\om)$ and $g_0 \in D(\theta_t \om)$, by the uniform Gronwall inequality applied to \eqref{snd50} and by \eqref{snd51}, \eqref{snd53} and \eqref{snd54}, we obtain
		\beq \bl{snd55}
		\begin{split}
			&\|\nb G(t,\om;\tau,g_0)\|^2  \leq e^{R_1 (\om)} \left\{(C_2 + 32\, \eta) R_1(\om) \right.\\[3pt]
			\quad & \left. + \int_{-3}^{0} \left[C_1 \left(\|\nb \G^h(\theta_s \om)\|^4 + 1\right) + 2 Q_2(s, \om) + C_2(h) |\G(\theta_s \om)|^2 \right]\,ds \right\}.
		\end{split}
		\eeq
Finally, by \eqref{snd31} and \eqref{snd55}, we reach the conclusion that any for $t < -T(D,\om)$,
		\beq \bl{snd56}
		\begin{split}
			 &\,|||\Phi (t, \theta_{-t} \,\om, D(\theta_{-t}\, \om)|||_E^2 =  \sup_{g_0 \in D(\theta_t \om)} \|\Phi (t, \theta_{-t}\, \om, g_0)\|_E^2 \\
			 = &\, \sup_{g_0 \in D(\theta_t \om)}  \|G(0, \theta_{-t} \om; \, -t, \, g_0) + \G^h(\om)\|_E^2\\
			 \leq &\, \sup_{g_0 \in D(\theta_t \om)} 2\left(\|G(0, \theta_{-t}\om; \, -t, \, g_0)\|_E^2 + \|\G^h(\om)\|_E^2\right)\\
			 = &\, \sup_{g_0 \in D(\theta_t \om)} 2\left(\|G(0,\theta_{-t} \om; \, -t, \, g_0)\|_H^2 + \|\nb G(0, \theta_{-t} \om \, ;-t, \, g_0)\|_H^2 +  \|\G^h(\om)\|_E^2\right)\\
			 \leq &\, R_E^2(\om),
		\end{split}
		\eeq
		where 
		\beq \bl{snd57}
		\begin{split}
			&R_E^2(\om) = 2Q_1 (\om) + 2\|\G^h(\om)\|_E^2 + 2e^{R_1 (\om)} \left\{(C_2 + 32\, \eta) R_1 (\om) \right.\\[3pt]
		 &\left. + \int_{-3}^{0} \left[C_2 \left(\|\nb \G^h(\theta_s \om)\|^4 + 1\right) + 2 Q_2(s, \om) + C_1 |\G(\theta_s \om)|^2 \right]\,ds \right\},
		\end{split}
		\eeq
		and $Q_1 (\om)$ is given in \eqref{Q1}. Note that $R_E (\om)$ is  a random variable independent of any initial time and initial state. Thus the result \eqref{snd36} of this theorem is proved.
\end{proof}

We complete this section by proving the main result on the existence of a random attractor for the Hindmarsh-Rose random dynamical system $\Phi$ in the space $H$.
\begin{theorem} \bl{lm5}
	For the spacial domain dimension $n = \textup{dim} \,(\gw) \leq 2$ and any positive parameters $d_1, d_2, d_3, a, b, \alpha, \beta, q$, $r, J$ and for any $c \in \mathbb{R}$, there exists a unique random attractor $\mathcal{A}(\om)$ in the space $H = L^2 (\gw, \mathbb{R}^3)$ with respect to $\mathscr{D}_H$ for the Hindmarsh-Rose random dynamical system $\Phi$ over the metric dynamical system $(\mathfrak{Q}, \mathcal{F}, P, \{{\theta_t}\}_{t \in \mathbb{R}})$. 
\end{theorem}

\begin{proof}
	In Theorem \ref{Thm2}, we proved that there exists a pullback absorbing set $K(\om) \subset H$ for the stochastic Hindmarsh-Rose cocycle $\Phi$. According to Definition \ref{sc6}, Theorem \ref{Thm3} and the compact imbedding $E \hookrightarrow H$ confirmed that this cocycle $\Phi$ is pullback asymptotically compact on $H$ with respect to $\mathscr{D}_H$. 
	
	Hence, by Theorem \ref{sc8}, there exists a unique random attractor in the space $H$ for this Hindmarsh-Rose random dynamical system $\Phi$, which is given by
	\beq \bl{srd58}
	\mathcal{A} (\omega) = \bigcap_{\tau \geq 0} \; {\overline{\bigcup_{ t \geq \tau} \varphi (t, \theta_{-t} \omega, K (\theta_{-t} \omega))}}, \quad \om \in \mathfrak{Q},
	\eeq
	where $K(\om) = B_H(0, R_H(\om))$ is defined in \eqref{snd34}. The proof is completed.
\end{proof}

	We make a remark that there is an essential difficulty in proving the pullback asymptotic compactness of the stochastic Hindmarsh-Rose cocycle for the space dimension $n=3$. This is the reason that we reduce the space dimension $n= \text{dim}\, (\gw) \leq 2$ in Theorem \ref{Thm3} and Theorem \ref{lm5} for this Hindmarsh-Rose random dynamical system. All the results shown in the Section 2 remain valid for space dimesion $n= \text{dim} \,(\gw) \leq 3$. We conjecture that there should exist a random attractor for the random dynamical system generated by the stochastic Hindmarsh-Rose equations with the additive noise also on the 3-dimensional domain space.

\vspace{10pt}
\bibliographystyle{amsplain}

\end{document}